\documentclass[10pt]{amsart}
\usepackage{amscd,amssymb}
\usepackage[all]{xy}
\usepackage{color}

\begin{document}
\newtheorem{prop}{Proposition}[section]
\newtheorem{thm}[prop]{Theorem}
\newtheorem{lemma}[prop]{Lemma}
\newtheorem{cor}[prop]{Corollary}
\newtheorem{Question}[prop]{Question}

\newtheorem{Example}[prop]{Example}
\newtheorem{Examples}[prop]{Examples}
\newtheorem{Remark}[prop]{Remark}

\newcommand{\extto}{\xrightarrow}
\newcommand{\cA}{{\mathcal A}}
\newcommand{\cB}{{\mathcal B}}
\newcommand{\cD}{{\mathcal D}}
\newcommand{\cY}{{\mathcal Y}}
\newcommand{\cU}{{\mathcal U}}
\newcommand{\cJ}{{\mathcal J}}
\newcommand{\cI}{{\mathcal I}}
\newcommand{\cF}{{\mathcal F}}
\newcommand{\cX}{{\mathcal X}}
\newcommand{\cZ}{{\mathcal Z}}
\newcommand{\cP}{{\mathcal P}}
\newcommand{\cQ}{{\mathcal Q}}
\newcommand{\cL}{{\mathcal L}}
\newcommand{\cM}{{\mathcal M}}
\newcommand{\cK}{{\mathcal K}}
\newcommand{\cH}{{\mathcal H}}
\newcommand{\cT}{{\mathcal T}}
\newcommand{\cG}{{\mathcal G}}
\newcommand{\cE}{{\mathcal E}}
\newcommand{\cC}{{\mathcal C}}
\newcommand{\cO}{{\mathcal O}}
\newcommand{\dpp}{\prime\prime}
\newcommand{\br}{{\bf r}}
\newcommand{\bx}{$_{\fbox{}}$\vskip .2in}
\newcommand{\Gr}{\operatorname{\mathbf{Gr}}\nolimits}
\newcommand{\Span}{\operatorname{Span}\nolimits}
\newcommand{\gr}{\operatorname{\mathbf{gr}}\nolimits}
\newcommand{\Mod}{\operatorname{\mathbf{Mod}}\nolimits}
\newcommand{\smod}{\operatorname{\mathbf{mod}}\nolimits}
\newcommand{\add}{\operatorname{\mathbf{add}}\nolimits}
\newcommand{\Add}{\operatorname{\mathbf{Add}}\nolimits}
\newcommand{\End}{\mbox{End}}
\newcommand{\Hom}{\mbox{Hom}}
\renewcommand{\Im}{\mbox{Im}}
\newcommand{\pd}{\operatorname{pd}\nolimits}
\newcommand{\id}{\operatorname{id}\nolimits}
\newcommand{\soc}{\operatorname{Soc}\nolimits}
\newcommand{\Top}{\operatorname{Top}\nolimits}
\newcommand{\Ker}{\mbox{Ker}}
\newcommand{\Coker}{\mbox{Coker}}
\newcommand{\coker}{\mbox{coker}}
\newcommand{\Tor}{\mbox{Tor}}
\renewcommand{\dim}{\mbox{dim}}
\newcommand{\gldim}{\operatorname{gl.dim}\nolimits}
\newcommand{\Ext}{\operatorname{Ext}\nolimits}
\newcommand{\op}{^{\mbox{op}}}
\newcommand{\pr}{^{\prime}}
\newcommand{\f}{\operatorname{fin}}
\newcommand{\Sy}{\operatorname{syz}}
\newcommand{\semi}{\mathbin{\vcenter{\hbox{$\scriptscriptstyle|$}}
\;\!\!\!\times }}
\newcommand{\fralg}{K\!\!<\!\!x_1,\dots,x_n\!\!>}
\newcommand{\Efin}{\operatorname{Efin}\nolimits}
\newcommand{\Lin}{\operatorname{Lin}\nolimits}
\newcommand{\Tlin}{\operatorname{Tlin}\nolimits}
\newcommand{\WKS}{\operatorname{WKS}\nolimits}
\newcommand{\cx}{\operatorname{cx}\nolimits}
\newcommand{\pcx}{\operatorname{pcx}\nolimits}
\newcommand{\tpcx}{\operatorname{tpcx}\nolimits}
\newcommand{\mpcx}{\operatorname{mpcx}\nolimits}
\newcommand{\tmpcx}{\operatorname{tmpcx}\nolimits}
\newcommand{\mto}{\hookrightarrow}
\newcommand{\AR} {Auslander-Reiten }
\newcommand{\gk}{\operatorname{GKdim}\nolimits}

\newcommand{\rank}{\operatorname{rank}\nolimits}

\newcommand{\rep}[1]{\operatorname{\mathbf{Rep}}\nolimits(K,{#1})}
\newcommand{\srep}[1]{\operatorname{\mathbf{rep}}\nolimits(K,{#1})}
\newcommand{\Ag}{\cA_{\Gamma}}
\newcommand{\Agw}{\cA_{\Gamma_W}}

\title[Comparison Theorem]
{Cohomological Comparison Theorem}

\author[Green]{Edward L.\ Green}\address{Department of
Mathematics\\ Virginia Tech\\ Blacksburg, VA 24061\\
USA} \email{green@math.vt.edu}
\author[Madsen]{Dag Oskar Madsen}
\address{Dag Oskar Madsen\\
University of Nordland\\ Faculty of Professional Studies} \email{Dag.Oskar.Madsen@uin.no}

\author[Marcos]{Eduardo Marcos}
\address{Eduardo Marcos\\
Departmento Matem\'atica \\Universidade de S\~ao Paulo \\
Brasil} \email{enmarcos@ime.usp.br}

\thanks{The third autor was partially supported by a research grant of CNPq-Brazil, (bolsa de pesquisa), and also from a tematic grand from  Fapesp- S\~ao Paulo, Brazil\\ This work was mostly done during some visits to Virginia Tech}

 \subjclass[2010]{Primary 16W50, 16E30. Secondary 16G10}

\begin{abstract} {If $f$ is an idempotent in a ring 
$\Lambda$, then we find sufficient
\linebreak conditions which imply
that  the cohomology rings 
$\oplus_{n\ge 0}\Ext^n_{\Lambda}(\Lambda/{\br},\Lambda/{\br})$ and \linebreak
$\oplus_{n\ge 0}\Ext^n_{f\Lambda f}(f\Lambda f/f{\br} f,f\Lambda f/f{\br} f)$  are eventually isomorphic.  This result allows us to compare finite generation and GK dimension of the cohomology rings $\Lambda$ and $f\Lambda f$.  We are also able to compare the global
dimensions of  $\Lambda$ and $f\Lambda f$.
}
\end{abstract}

\maketitle
\section{Introduction}

If $M$ is a $\Lambda$-module for some ring $\Lambda$, knowledge of the cohomology ring
of $M$, $\oplus_{n\ge 0}\Ext_{\Lambda}^n(M,M)$, is useful in the study of the representation
theory of $\Lambda$-modules.  In view of this,  connecting  the cohomology rings of
two modules over different rings can provide helpful information.     
The main goal of this paper is to find  sufficient conditions so that
the cohomology rings 
$\bigoplus_{n\ge 0}\Ext^n_{\Lambda}(\Lambda/{\br},\Lambda/{\br})$ and
$\bigoplus_{n\ge 0}\Ext^n_{f\Lambda f}(f\Lambda f/f{\br} f,f\Lambda f/f{\br} f)$  are eventually isomorphic, where $f$ is an idempotent in
the ring $\Lambda$ and $\br$, denotes the Jacobson radical of $\Lambda$.
Our results are stated in the more general setting of graded rings.
In \cite{dk}, conditions are found implying the full cohomology rings are isomorphic and
\cite{pos} contains results that are related to ours.

To properly summarize the contents of this paper, we introduce some definitions
and notation.
Let $G$ be a group and 
let $\Lambda=\oplus_{g\in G}\Lambda_{g}$ be a $G$-graded ring; that is,
if $g,h\in G$, then $\Lambda_g\cdot \Lambda_{h}
\subseteq \Lambda_{gh}$.    
We denote the identity of $G$ by $\mathfrak{e}$,  the graded Jacobson radical of $\Lambda$
by $\br$ and set $\br_{\mathfrak e}=\Lambda_{\mathfrak e}\cap \br$.
A $G$-grading on $\Lambda$ will be called a  \emph{proper
$G$-grading} when it satisfies the following conditions:
 if $g\ne \mathfrak e$ then $\Lambda_g\cdot
\Lambda_{g^{-1}}\subseteq \br_{\mathfrak e}$ 
and  $\Lambda_{\mathfrak e}/\br_{\mathfrak e}$ is a 
semisimple Artin algebra over a commutative Artin ring $C$. 

For a proper $G$ grading it is easy to see that
 that $\br=\br_{\mathfrak e}\oplus(\oplus_{g\in G}\Lambda_g)$ and that
$\br_{\mathfrak e}$ is the intersection of the left maximal
graded ideals in $\Lambda_{\mathfrak e}$
that contain $\sum_{g\in G\setminus\{\mathfrak e\}}\Lambda_g\cdot \Lambda_{g^{-1}}$. 
We view $\Lambda$ as a $G$-graded $\Lambda$-module with $\Lambda_g$ living
in degree $g$.

We also fix the following notation:
$\Mod(\Lambda)$ will denote the category of left $\Lambda$-modules, $\Gr(\Lambda)$ the category of graded $\Lambda$-modules, and $\gr(\Lambda)$ the full subcategory of $\Gr(\Lambda)$ consisting of the finitely generated graded modules.
Given  a $\Lambda$ module $X$,  we let $\pd_{\Lambda}(X)$ and $\id_{\Lambda}(X)$ denote the projective dimension and the injective dimension of $X$ over $\Lambda$ respectively.
If \[\cP:\cdots \to P^2\extto{\delta^2}P^1\extto{\delta^1}P^1\extto{\delta^1}P^0\to X\to
0\] is a graded projective $\Lambda$-resolution of a graded module $X$, then $\cP^{>c}$ will denote the 
resolution of the $\Omega^n(X) =\Im(P^n\to P^{n-1})$, obtained from $\cP$.
The main comparison is Theorem \ref{grG-to-ss-thm} which we state below, omitting
some technicalities.  

\newtheorem*{Thm}{Theorem}
\begin{Thm}
 Let $G$ be a group
and $\Lambda=\oplus_{g\in G}\Lambda_g$ be
a properly $G$-graded ring.
Suppose that  $e$ is an idempotent  in
$\Lambda$, $f=1-e$ and set $\Lambda^*$ be the ring $f\Lambda
f$ and $\br^*=f\br f$. Assume that  $\pd_{\Lambda^*}(f\Lambda e)=c<\infty$, 
$\pd_{\Lambda}((\Lambda/\br)  e)=a<\infty$, and
that $\id_{\Lambda}((\Lambda/\br) e)=b<\infty$. Then, for $n>\max\{a,b+c+2\}$,
there are isomorphisms
$\Ext^n_{\Lambda}(\Lambda/\br,\Lambda/\br)\cong
\Ext^n_{\Lambda^*}(\Lambda^*/\br^*,\Lambda^*/\br^*)$ such that the induced
isomorphism
\[\bigoplus_{n>\max\{a,b+c+2\}}\Ext^n_{\Lambda}(\Lambda/\br,\Lambda/\br)\cong
\bigoplus_{n>\max\{a,b+c+2\}}\Ext^n_{\Lambda^*}(\Lambda^*/\br^*,\Lambda^*/\br^*)\]
\end{Thm}

We also obtain the following applications; see Theorem \ref{bigthm}.  To simplify
notation, we write $E(\Lambda)$ for the cohomology ring
$\oplus_{n\ge 0}\Ext^n_{\Lambda}(\Lambda/\br,\Lambda/\br) $.

\begin{Thm} Keeping the hypotheses of the above Theorem,
the following hold.
\begin{enumerate}
\item Assume that $f\Lambda e$ has a finitely generated minimal
graded projective
$\Lambda^*$-resolution. 
The cohomology ring $E(\Lambda)$ is finitely generated over\linebreak
$\Ext^0_{\Lambda}(\Lambda/\br,\Lambda/\br)\cong \Hom_{\Lambda}(\Lambda/\br,\Lambda/\br)
\cong (\Lambda/\br)^{op}$ if and only if the cohomology ring $E(\Lambda^*)$ is finitely generated
as a $(\Lambda^*/\br^*)^{op}$-algebra. 
\item Assume that $\Lambda$ is $K$-algebra, where $K$ is a field and that $\Lambda/\br$ is a finite dimensional $K$-algebra.   Assume further that both $E(\Lambda)$ and $E(\Lambda^*)$
are finitely generated $K$-algebras.  Then
$\gk(E(\Lambda))
=\gk(E(\Lambda^*))$.
\item  We have that $\pd_{\Lambda}(S)<\infty$, for all
graded simple $\Lambda$-modules $S$ if and only if $\pd_{\Lambda^*}(S^*)<\infty$, for all
graded simple $\Lambda^*$-modules $S^*$.  
\end{enumerate}

\end{Thm}

\section{Comparison theorem}\label{main result}

Let $G$ be a group and 
let $\Lambda=\oplus_{g\in G}\Lambda_{g}$ be a $G$-graded ring; in particular,
if $g,h\in G$, then $\Lambda_g\cdot \Lambda_{h}
\subseteq \Lambda_{gh}$.    
We denote the identity of $G$ by $\mathfrak{e}$,  the graded Jacobson radical of $\Lambda$
by $\br$ and set $\br_{\mathfrak e}=\Lambda_{\mathfrak e}\cap \br$.
We assume throughout this section that if $g\ne \mathfrak e$ then $\Lambda_g\cdot
\Lambda_{g^{-1}}\subseteq \br_{\mathfrak e}$ 
and that $\Lambda_{\mathfrak e}/\br_{\mathfrak e}$ is a 
semisimple Artin algebra over 
a commutative Artin ring $C$. We call such a $G$-grading on $\Lambda$ a \emph{proper
$G$-grading}. 
It is easy to see that $\br=\br_{\mathfrak e}\oplus(\oplus_{g\in G}\Lambda_g)$ and that
$\br_{\mathfrak e}$ is the intersection of the left maximal
{\tt graded} ideals in $\Lambda_{\mathfrak e}$
that contain $\sum_{g\in G\setminus\{\mathfrak e\}}\Lambda_g\cdot \Lambda_{g^{-1}}$. 
We view $\Lambda$ as a $G$-graded $\Lambda$-module with $\Lambda_g$ living
in degree $g$.

We denote the category of 
$G$-graded 
$\Lambda$-modules and degree $\mathfrak e$ maps by $\Gr(\Lambda)$. We let $\gr(\Lambda)$
denote the full subcategory of finitely generated $G$-graded $\Lambda$-modules and  
let $\Lambda=\oplus_{g\in G}\Lambda_g$ be a properly $G$-graded ring.   The shifts by elements of $G$ induce a group of endofunctors on
$\Gr(\Lambda)$.  More precisely, the shift functor associated to an element $h\in G$ is defined
as follows: if $X=\oplus_{g\in G}X_g$ is a graded
$\Lambda$-module, we let
$X[h]=\oplus_{g\in G}Y_g$, where $Y_g=X_{hg}$.  Let $\Phi\colon \Gr(\Lambda)\to \Mod(\Lambda)$ denote the forgetful functor.
  
If $X\in \gr(\Lambda)$ and $Y$ is a graded $\Lambda$-module then
\[ \Hom_{\Lambda}(\Phi(X),\Phi(Y))=\Hom_{\Gr(\Lambda)}(X,\oplus_{g\in G}Y[g])\cong
\bigoplus_{g\in G}\Hom_{\Gr(\Lambda)}(X,Y[g]).
\]

We need one further assumption; namely, if $\overline \epsilon$ is an idempotent element in
$\Lambda_{\mathfrak e}/\br_{\mathfrak e}$, then there is an idempotent $\epsilon\in 
\Lambda_{\mathfrak e}$ such that $\pi(\epsilon)=\overline{\epsilon}$, where $\pi\colon \Lambda_{\mathfrak e}
\to \Lambda_{\mathfrak e}/\br_{\mathfrak e}$ is the canonical surjection.  If a graded
ring $\Lambda$ has this property, we say \emph{graded idempotents lift}. Assume that graded
idempotents lift in $\Lambda$. It follows that
if $S$ is simple graded $\Lambda$-module, then 
$S\cong (\Lambda_{\mathfrak e}/\br_{\mathfrak e})
\epsilon[g]$, for some primitive idempotent $\epsilon\in \Lambda_{\mathfrak e}$ and some
$g\in G$.  We also see that the canonical surjection $\Lambda \epsilon[g]\to (\Lambda_{\mathfrak e}/\br_{\mathfrak e})\epsilon[g]$ is a projective cover.
  The next three examples provide 
important classes of graded rings satisfying our assumptions.
2
\begin{Example}\label{Gfinite}{\rm  Let $K$ be a field, $\cQ$ a finite quiver, $G$ a group,
and $W\colon \cQ_1\to G\setminus\{\mathfrak e\}$. We call $W$ a \emph{weight function};
see \cite{gr}.  Setting $W(v)=\mathfrak e$ for all vertices
$v$ in  $\cQ$, and, if $p=a_1\cdots a_n$ is a path of length $n\ge 1$, with the $a_i\in\cQ_1$, then
set $W(p)=W(a_n)W(a_{n-1})\cdots W(a_1)$. In this case, we say $p$ has \emph{weight $W(p)$}.
 We $G$-grade the path algebra $K\cQ$ by defining $K\cQ_g$ to be the $K$-span of
paths $p$ of weight $g$.  Let $I$ be an ideal in $K\cQ$ such that $I$ can be generated
by elements $x_i$, such that, for each $i$, the paths occuring in $x_i$ are all of length at
least 2 and all have the 
same weight.  Let $\Lambda= K\cQ/I$.  The $G$-grading on $K\cQ$ induces a $G$-grading
on $\Lambda$.  

Note that if $a\in\cQ_1$ with $W(a)=g$, then $a+I$ is a nonzero element in $\Lambda_g$.  Using
that $g\ne\mathfrak e$, one can show that $\br$ is the ideal generated by $\{a+I\mid
a\in\cQ_1\}$.  It follows that $\Lambda/\br$ is the semisimple ring $\prod_{v\in\cQ_0}K$, which
is a semisimple Artin algebra over $K$.  Furthermore, one may check that $\br_{\mathfrak e}$
is the ideal in $\Lambda_{\mathfrak e}$ generated by the elements of the form $p+I$, where
$p$ is a path of length $\ge 1$ in $\cQ$ of weight $\mathfrak e$.  Thus the $G$-grading on $\Lambda$
is a proper $G$-grading.  It is also clear that graded idempotents lift. \qed
}\end{Example}

\begin{Example}\label{pos-Z}{\rm  Let $G=\mathbb Z$ and
let $\Lambda =\Lambda_0\oplus \Lambda_1\oplus \Lambda _2\oplus \cdots$ be a positively $\mathbb Z$-graded
ring such that $\Lambda_0$ is an Artin algebra.   It is immediate that $\Lambda$ is a properly 
$\mathbb Z$-graded ring in which graded idempotents lift.  

}\end{Example}

\begin{Example}\label{artin-alg}{\rm  Let 
$\Lambda$ be an Artin algebra  over a commutative Artin ring $C$.   Let $G$ be any group and 
$\Lambda_{\mathfrak e}=\Lambda$, and, for $g\in G\setminus\{\mathfrak e\}$, 
$\Lambda_g=0$.  We see that $\Lambda$, as a $G$-graded ring, is properly $G$-graded and graded
idempotents lift.  One choice for $G$ is the trivial group $\{\mathfrak e\}$.

}\end{Example}

We recall 
some known results about graded projective resolutions over properly graded rings in which
graded idempotents lift.   We leave the proof to the reader.  

\begin{lemma}\label{proj-cov} Let $\Gamma=\oplus_{g\in G}\Gamma_g$ be a properly
$G$-graded ring in which graded idempotents lift and let $\br_{\Gamma}$ denote
the graded Jacobson radical of $\Gamma$. Suppose $X$ is a finitely 
generated graded $\Gamma$-module and $X/\br_{\Gamma}X\cong\oplus_{i=1}^nS_i$, where each
$S_i$ is a graded simple $\Gamma$-module.  Let $P_i\extto{\alpha_i}S_i$ be graded projective covers
for each $i$ and let $P=\oplus_{i=1}^nP_i$.  Then
\begin{enumerate}
\item For each $i=1,\dots,n$, $P_i\cong \Gamma \epsilon_i[g]$, for some idempotent 
$\epsilon_i\in \Gamma$ and $g\in G$.
\item The map $P\extto{\oplus_{i=1}^n\alpha_i}\oplus_{i=1}^nS_i$ is a graded projective cover.
\item If $P\extto{\beta}X$ is a graded map such that the following diagram commutes
\[
\xymatrix{ X\ar^{\pi\phantom{xxx}}[r]&X/\br_{\Gamma}X\\
P\ar^{\beta}[u]\ar^{\oplus_i\alpha_i\phantom{xx}}[r]&\oplus_{i=1}^nS_i,\ar@{=}[u],
}\]
where $\pi$ is the canonical surjection,
then $\beta\colon P\to X$ is a graded projective cover.  Moreover, $\ker(\beta)\subseteq
\br_{\Gamma}P$.
\item If $0\to K\extto{\sigma} P\extto{\beta} X\to 0$ is a short exact sequence in $\Gr(\Gamma)$
with $P$ finitely generated,
such that   $\sigma(K)\subseteq \br_{\Gamma}P$, then $\beta$ is a graded projective cover.
\item Suppose that 
\[\cP:\cdots \to P^2\extto{\delta^2}P^1\extto{\delta^1}P^1\extto{\delta^1}P^0\to X\to
0\] is a graded projective $\Gamma$-resolution of $X$ with each $P^n$ finitely generated.
Then $\cP$ is minimal if and only if, for $n\ge 1$, $\delta^n(P^n)\subseteq \br_{\Gamma}P^{n-1}$.
\item If $P$ and $Q$ are finitely generated graded projective $\Gamma$-modules and
$\alpha\colon  P \to Q$ is a map in $\gr(\Gamma)$,  then there are primitive idempotents $\epsilon_i$ and $\epsilon'_{j}$
and elements $g_i$ and $h_i$ of $G$, $i=1,\dots,m$ and $j=1,\dots,n$, for some integers
$m$ and $n$ such that
\begin{enumerate}
\item $P\cong\oplus_{i=1}^m\Gamma \epsilon_i[g_i]$,
\item $Q\cong\oplus_{j=1}^n\Gamma \epsilon_j[h_j]$, and
\item viewing (a) and (b) as identifications, $\alpha$ is given by an $m\times n$ matrix $(\gamma_{i,j})$, where  $\gamma_{i,j}\in
\epsilon_i\Gamma_{h_jg_i^{-1}}\epsilon_j$.
\end{enumerate}
\item Keeping the notation and assumptions of part (5), we see that $\cP$ is a minimal graded projective
resolution of $X$ if and only if the matrices that give the $\delta^n$, $n\ge 0$, all have entries in $\br_{\Gamma}$.
\item The forgetful functor $\Phi$ is exact, preserves direct sums, and, if $Y$ is a graded
$\Gamma$-module,   $\Phi(Y)$ is a projective $\Gamma$-module
if and only if $Y$ is a graded projective $\Gamma$-module.  Thus, $\Phi$ takes graded projective $\Gamma$-resolutions to projective $\Gamma$-resolutions.

\end{enumerate}\qed
\end{lemma} 

Let $e$ be an idempotent in $\Lambda_{\mathfrak e}$.  We say that $(e,f)$ is
a \emph{suitable idempotent pair} if $f=1-e$ and $f\Lambda e\subseteq
\br$.  Note that if $(e,f)$ is a suitable idempotent pair, then,
since $e$ and
$1$ are homogeneous of degree $\mathfrak e$, so is $f$.  Furthermore,
if $(e,f)$ is a suitable idempotent pair, then $\Hom_{\Lambda}(
(\Lambda/\br)e,(\Lambda/\br)f)=  \Hom_{\Lambda}(
(\Lambda/\br)f,(\Lambda/\br)e)= 0$.  Note that if $(e,f)$ is a
suitable idempotent pair, then $(f,e)$ is also a suitable  idempotent
pair..

For the remainder of this section, we fix a suitable
idempotent pair $(e,f)$.
 Let $\Lambda^*=f\Lambda f$ and $\br^*=f\br f$.  The $G$-grading of
$\Lambda$ induces a $G$-grading on $\Lambda^*$ and it is not hard to show
that $\br^*$ is the graded Jacobson radical
of $\Lambda^*$.  

The main tool in this section is the functor $F\colon \Gr(\Lambda)\to \Gr(\Lambda^*)$  given
by $F(X)=f\Lambda\otimes_{\Lambda}X$.  Let $H\colon \Gr(\Lambda^*)\to \Gr(\Lambda)$ be given
by $H(X)=\Hom_{\Lambda^*}(f\Lambda ,X)$. Note that both $F(X)$ and  $H(X)$ have induced $G$-gradings
obtained from the gradings of $X$ and $\Lambda f$.
   The next result is well-known.

\begin{prop}\label{FGbasic} Keeping the above notation, we have that
\begin{enumerate}
\item the functor $F$ is exact,
\item $(F,H)$ is an adjoint pair, and
\item $f\Lambda\cong \Lambda^*\oplus f\Lambda e$, as left
$\Lambda^*$-modules.
\end{enumerate}\qed
\end{prop}

The functor $H$ is exact if and only if $f\Lambda$ is a left
projective $\Lambda^*$-module, and, by Proposition \ref{FGbasic}(3),
$H$ is exact if and only if $f\Lambda e$ is a left projective
$\Lambda^*$-module. Note that $F(\Lambda e)\cong f\Lambda e$ does not, in general,
have finite projective dimension as a left $\Lambda^*$-module, as the example below 
demonstrates.

\begin{Example}{\rm  Let $\cQ$ be the quiver
\[
\xymatrix{
\stackrel{u}{\circ}\ar_a[r]&\stackrel{v}{\circ}\ar@(r,u)[]^b
}
\]
Let $I$ be the ideal generated by $ba$ and $b^2$ and let
$\Lambda=\cQ/I$.  Taking $e=u$ and $f=v$, we see
that $f\Lambda e$ has infinite projective dimension viewed
as a left $\Lambda^*$-module where $\Lambda^*=f\Lambda f$.
}\end{Example}

We note that if $X$ is a graded $\Lambda$-module, then $F(\Phi(X))\cong \Phi(F(X))$ and if $\cP:\cdots \to P^2\extto{\delta^2}P^1\extto{\delta^1}P^0\to X\to
0$ is a graded projective resolution with syzygies $\Omega^n_{\Lambda}(X)$ then
$\Phi(\cP)$ is projective resolution of $\Phi(X)$,
\[ F(\Phi(\Omega^n_{\Lambda}(X))\cong  \Phi(\Omega^n_{\Lambda}(F(X)),\]
where $\Omega^n_{\Lambda}(F(X))$ denotes the $n$-th syzygy of $F(\cP)$.

We abuse notation by denoting the forgetful functor from $\Gr(\Lambda^*)$ to
$-(\Lambda^*)$ also by $\Phi$.   We also use $F$ to denote the functor  
$f\Lambda \otimes_{\Lambda}-$ from $\Mod(\Lambda)$ to $\Mod(\Lambda^*)$.   The
meaning of both $F$ and $\Phi$ will be clear from the 
context.
 
The next result is quite general and will allow us to apply the functor $F$ and keep
control of the cohomology if $\pd_{\Lambda^*}(f\Lambda e)<\infty$.
 One does not need that the $G$-grading is proper.

\begin{thm}\label{grapplyF}
Let $G$ be a group with identity element $\mathfrak e$ 
and $\Lambda$ be a $G$-graded ring  and
let $(e,f)$ be a suitable idempotent pair in $\Lambda$.
Set $\Lambda^*=f\Lambda f$.
 Suppose that $\pd_{\Lambda^*}(f\Lambda e)=c<\infty$. Let $X$ be
a graded left $\Lambda$-module and
$\Omega^i_{\Lambda}(X)$ (respectively, $\Omega^i_{\Lambda^*}(F(X))$)
denote the $i$-th sygyzy of $X$ (resp., $F(X)$) in a graded projective
$\Lambda$-resolution of $X$ (resp., a graded projective
$\Lambda^*$-resolution of $F(X)$). Then, for $t>c+1$ and $n\ge 0$,
\[
\Ext_{\Lambda^*}^t(\Phi(F(\Omega^n_{\Lambda}(X))),-)\cong
\Ext_{\Lambda^*}^t(\Phi(\Omega^n_{\Lambda^*}(F(X))),-).
\]

\end{thm}

\begin{proof} For $n=0$ the result is clear and hence we assume $n\ge 1$. Without loss of 
generality,  we may start with a graded $\Lambda$-resolution of $X$ in which
each graded projective module is a direct sum of copies of graded projective
modules of the form $\Lambda[g]$, for $g\in G$.
Since $1=e+f$, this resolution has the form:  
\[ \cdots\to P^2\oplus Q^2\to P^1\oplus Q^1\to P^0\oplus Q^0\to X\to
0,\] where $P^i$ is a direct sum of copies of  graded modules of the
form $\Lambda f[g]$ and $Q^i$
is a direct sum of copies of graded modules of the form $\Lambda e[g]$, for $i\ge 0$. Setting
$F(P^i)=L^i$ and $F(Q^i)=M^i$, we note that $L^i$ is a graded projective
$\Lambda^*$-module and $M^i$ is a direct sum of copies of graded modules of the form 
$(f\Lambda e)[g]$. Applying the exact functor $F$ to the resolution above, we
obtain an exact sequence of graded $\Lambda^*$-modules
\[ \cdots\to L^2\oplus M^2\to L^1\oplus M^1\to L^0\oplus M^0\to F(X)\to
0. \] For $i\ge 1$, note $F(\Omega^i_{\Lambda}(X))=\Im(L^i\oplus
M^i\to L^{i-1}\oplus M^{i-1})$ and $L^i$ is a graded left projective
$\Lambda^*$-module. For ease of notation, we let
$Z_i=F(\Omega^i_{\Lambda}(X))$, for $i\ge 1$ and $Z_0=F(X)$.

For $n\ge 1$, we have a short exact sequence of graded $\Lambda^*$-modules
\[
0\to Z_n\to L^{n-1}\oplus M^{n-1}\to Z_{n-1}\to 0.\]

Let $ P(M^{n-1})\to M^{n-1}\to 0$ be exact sequence of graded
$\Lambda^*$-modules with $ P(M^{n-1})$ a graded projective module. Then we
obtain the following exact commutative diagram:
\[
\xymatrix{ &0&0\\
0\ar[r]&Z_n\ar[u]\ar[r]&L^{n-1}\oplus
M^{n-1}\ar[u]\ar[r]\ar[u]&Z_{n-1}\ar[r] &0\\
0\ar[r]&\Omega^1_{\Lambda^*}(Z_{n-1})\ar[u]\ar[r]&L^{n-1}\oplus
P(M^{n-1})\ar[u]\ar[r]&Z_{n-1}\ar[r]\ar[u]^= &0\\
0\ar[r]&\Omega^1_{\Lambda^*}(M^{n-1})\ar[u]\ar[r]^= &
\Omega^1_{\Lambda^*}(M^{n-1})\ar[u]\\
 &0\ar[u]&0\ar[u]
}\] The first column yields the short exact sequence \[ 0\to
\Omega^1_{\Lambda^*}(M^{n-1}) \to \Omega^1_{\Lambda^*}(Z_{n-1})\to
Z_n\to 0.\]   Taking graded projective $\Lambda^*$-resolutions of the two
end modules, applying the Horseshoe lemma, and taking syzygies,  we obtain short exact sequences
\[
\quad 0\to \Omega^{j+1}_{\Lambda^*}(M^{n-1}) \to
\Omega^{j+1}_{\Lambda^*}(Z_{n-1})\to \Omega_{\Lambda^*}^j(Z_n)\to
0,\]
for $j\ge 0$.
Hence we obtain short exact sequences of $\Lambda^*$-modules
\[
\quad 0\to \Phi( \Omega^{j+1}_{\Lambda^*}(M^{n-1})) \to
\Phi(\Omega^{j+1}_{\Lambda^*}(Z_{n-1}))\to \Phi( \Omega_{\Lambda^*}^j(Z_n))\to
0.\]

 Note that $\Phi(\Omega^{j+1}_{\Lambda^*}(M^{n-1})))$ a projective
$\Lambda^*$-module if $j\ge c$ since \linebreak
$c\ge\pd_{\Lambda^*}(\Phi(M^{n-1}))$.
It follows that, for $j\ge c$ and $t\ge 2$,
\[
\Ext^{t+j}_{\Lambda^*}(\Phi(Z_n),-)\cong
\Ext^{t}_{\Lambda^*}(\Phi(\Omega_{\Lambda^*}^j(Z_n)),-)\cong
\Ext^{t}_{\Lambda^*}(\Phi(\Omega_{\Lambda^*}^{j+1}(Z_{n-1})),-) \cong\]\[
\cong\Ext^{t}_{\Lambda^*}(\Phi(\Omega_{\Lambda^*}^{j+2}(Z_{n-2})),-)\cong
\cdots\cong \Ext^{t}_{\Lambda^*}(\Phi(\Omega_{\Lambda^*}^{j+n}(Z_{0})),-)\cong\]\[
\cong \Ext^{t +j}_{\Lambda^*}(\Phi(\Omega_{\Lambda^*}^{n}(Z_{0})),-) .\]
Finally, we note that $Z_n=F(\Omega^n_{\Lambda}(X))$ and $Z_0= F(X)$
and the result follows.
\end{proof}

The next result is immediate and we only provide a sketch of the proof.

\begin{prop}\label{grExt-G-S}Let $G$ be a group and $\Lambda$ a properly
$G$-graded ring with graded Jacobson radical $\br$ and suitable idempotent
pair $(e,f)$.
We have that
 $\id_{\Lambda}((\Lambda/\br)e)\le a<\infty$ if and only if, for
every graded $\Lambda$-module $X$,
\[
\bigoplus_{n>a}\Ext^n_{\Lambda}(\Phi(X),(\Lambda/\br)f)
\cong
\bigoplus_{n>a}\Ext^n_{\Lambda}(\Phi(X),\Lambda/\br)
\]  as $\mathbb Z\times G$-graded modules over the
$\mathbb Z\times G$-graded ring $\oplus\Ext^n_{\Lambda}(\Lambda/\br,\Lambda/\br)$.  
Furthermore,
 $\pd_{\Lambda}((\Lambda/\br)e)\le a<\infty$,
 $\id_{\Lambda}((\Lambda/\br)e)\le a<\infty$ if and only if
\[
\bigoplus_{n>a}\Ext^n_{\Lambda}((\Lambda/\br)f,(\Lambda/\br)f)
\cong
\bigoplus_{n>a}\Ext^n_{\Lambda}(\Lambda/\br,\Lambda/\br)
\]  as $\mathbb Z\times G$-graded rings without
identity. \qed
\end{prop}

\begin{proof}
Since $\Lambda/\br=\Lambda_0\cong \Lambda_0e\oplus \Lambda_0f$,
\[\Ext^i_{\Lambda}(X,\Lambda/\br)=
\Ext^i_{\Lambda}(X,(\Lambda/\br)e)\oplus\Ext^i_{\Lambda}(X,(\Lambda/\br)f)\] and

\[\Ext^i_{\Lambda}(\Lambda_0,\Lambda_0)=
\Ext^i_{\Lambda}(\Lambda_0e,\Lambda_0e)\oplus\Ext^i_{\Lambda}(\Lambda_0e,\Lambda_0f)\oplus\Ext^i_{\Lambda}(\Lambda_0f,\Lambda_0e)
\oplus\Ext^i_{\Lambda}(\Lambda_0f,\Lambda_0f)\]
 the result follows.
\end{proof}

If $X$ is a graded  $\Lambda$-module and $\cP:\cdots \to P^2\extto{\delta^2}P^1\extto{\delta^1}P^1\extto{\delta^1}P^0\to X\to
0$ is a graded projective $\Lambda$-resolution of $X$, then we say that \emph{$\cP$ is finitely
generated} if $P^n$ is a finitely generated graded $\Lambda$-module for $n\ge 0$.

Let $\epsilon$ be an idempotent element of $\Lambda_{\mathfrak e}$. We say a graded simple module $S$
\emph{belongs to $\epsilon$} if $\epsilon S\ne 0$.  Equivalently, $S$ belongs to
$\epsilon$ if $S$ is isomorphic to a summand of $(\Lambda/\br)\epsilon [g]$, for some
$g\in G$.  We say a graded projective $\Lambda$-module $P$ \emph{belongs to $\epsilon$}, if
$P/\br P$ is a direct sum of  graded simple $\Lambda$-modules with each summand 
belonging to $\epsilon$.  We now state a useful result.  
  
\begin{lemma}\label{grFresol} Let $X$ be a graded  $\Lambda$-module and assume that
$\cP^{\bullet}:\cdots\stackrel{d^3}{\to} P^2\stackrel{d^2}{\to}
P^1\stackrel{d^1}{\to} P^0\stackrel{d^0}{\to}X\to 0$ is a graded projective
$\Lambda$-resolution of $X$ such that, for $n>c$, $P^n$ 
belongs to $f$. Then 
\begin{enumerate}
\item $F(\cP^{>c})$ is a projective $\Lambda^*$-resolution of
$F(\Omega^{c+1}X)$, where $(\Omega^{c+1}X)$ is $(c+1)$-st syzygy in
$\cP^{\bullet}$.
\item If $\cP^{\bullet}$ is a finitely generated minimal graded projective $\Lambda$-resolution 
of $X$, then $F(\cP^{\ge c+1})$ is
a finitely generated minimal graded projective $\Lambda^*$-resolution of
$F(\Omega^{c+1}X)$. 
\end{enumerate}
\end{lemma}

\begin{proof}  The functor $F$ is exact.  We need to
show that if $P$ belongs to $f$, then $F(P)$ is a
projective $\Lambda^*$-module. Since $P$ belongs to $f$,
$P$ is a direct sum of indecomposable projective modules, each 
of which is a summand of $(\Lambda f)[g]$, for some $g\in G$.  Thus, it suffices to show that,
for $g\in G$,
$F((\Lambda f)[g])$ is a graded projective $\Lambda^*$-module.  But $F((\Lambda f)[g])=
(f\Lambda\otimes_{\Lambda}\Lambda f)[g]\cong (f\Lambda f)[g]=\Lambda^*[g]$
and part (1) follows.

By minimality and our assumptions, the maps $F(d^i)$, viewed
as matrices (as in Proposition \ref{proj-cov}), have entries in $f\br f$.  But
$f\br f=\br^*$, the graded Jacobson radical of $\Lambda^*$, and (2) follows.

\end{proof}

The following is an immediate consequence of the above Lemma.

\begin{cor}\label{grF-to-resol}
Assume that $\id_{\Lambda}((\Lambda/\br)e)=b<\infty$.
Suppose that $X$ is a graded $\Lambda$-module and let
\[
\cP^{\bullet}:\cdots{\to} P^2\stackrel{d^2}{\to}
P^1\stackrel{d^1}{\to} P^0\stackrel{d^0}{\to}X\to 0 \] be a
minimal graded projective $\Lambda$-resolution of $X$. Then,
for $n>b$, $P^n$ belongs to $f$ and
$F(\cP^{>b})$ is a minimal graded projective $\Lambda^*$-resolution of
$F(\Omega^{b+1}_{\Lambda}(X))$. 
\end{cor}

\begin{proof} Let $n>b$ and consider $P^n$.  If there is an indecomposable
summand of $ P^n$ belonging to $e$,  then $\Ext^n_{\Lambda}(X,(\Lambda/\br)e)\ne 0$, contradicting   $\id_{\Lambda}((\Lambda/\br)e)=b$.
Hence, $P^n$ belongs to $f$ and the result follows.
\end{proof}

Using the above result we have one of the main results of this section.

\begin{thm}\label{grfirst-half} Let $G$ be a group
and $\Lambda=\oplus_{g\in G}\Lambda_g$ be
a properly $G$-graded ring in which graded idempotents lift.
Let $\br$ denote the graded Jacobson radical of $\Lambda$ and
 $(e,f)$ be a suitable idempotent pair.  Set $\Lambda^*$ be the ring $f\Lambda
f$ and $\br^*=f\br f$. Assume that $\pd_{\Lambda^*}f\Lambda e=c<\infty$, and
that $\id_{\Lambda}(\Lambda/\br e)=b<\infty$. Then, for a graded
$\Lambda$-module $X$ having finitely generated projective
resolutions and for $n>b+c+2$, the
functor $F=f\Lambda\otimes_{\Lambda}-\colon\Gr(\Lambda)\to
\Gr(\Lambda^*)$ induces isomorphisms
\[
\Ext^n_{\Lambda}(\Phi(X),(\Lambda/\br )f)\cong
\Ext^n_{\Lambda^*}(\Phi(F(X)),\Lambda^*/\br^*).
\]
Moreover,  assuming that every graded simple $\Lambda$-module belonging
to $f$ has a finitely generated minimal
graded projective resolution, then the induced isomorphism
\[
\bigoplus_{n>
b+c+2}\Ext^n_{\Lambda}((\Lambda/\br) f,(\Lambda/\br )f)\cong
\bigoplus_{n>
b+c+2}\Ext^n_{\Lambda^*}(\Lambda^*/\br^*,\Lambda^*/\br^*)\]
is an isomorphism of $\mathbb Z\times G$-graded rings without
identity.
Furthermore, identifying  
$\oplus_{n>
b+c+2}\Ext^n_{\Lambda}((\Lambda/\br) f,(\Lambda/\br )f)$ and
$\oplus_{n>
b+c+2}\Ext^n_{\Lambda^*}(\Lambda^*/\br^*,\Lambda^*/\br^*)$ and
denoting this ring by $\Delta$, $\oplus_{n>b+c+2}\Ext^n_{\Lambda}(\Phi(X),(\Lambda/\br )f)$ and
$\oplus_{n>b+c+2}\Ext^n_{\Lambda^*}(\Phi(F(X)),\Lambda^*/\br^*)$ are
isomorphic as graded $\Delta$-modules.
\end{thm}

\begin{proof} Let $X$ be a graded $\Lambda$-module and let \[\cP^{\bullet}:\cdots{\to} P^2\stackrel{\delta^2}{\to}
P^1\stackrel{\delta^1}{\to} P^0\stackrel{\delta^0}{\to}X\to 0
\] be a minimal graded projective $\Lambda$-resolution of the graded module
$X$. By our assumption that $id_{\Lambda}((\Lambda/\br)e)=b$, for
$n>b$, $P^n$ belongs to $f$. Hence,  
applying the functor $F$ and Proposition
\ref{grF-to-resol}, we see that
\[
F(\cP^{\ge n}):\cdots{\to} F(P^{n+2})\stackrel{F(\delta^{n+2})}{\to}
F(P^{n+1})\stackrel{F(\delta^{n+1})}{\to}
F(P^n)\stackrel{F(\delta^n)}{\to}F(\Omega^n(X))\to 0
\]
is a minimal graded projective $\Lambda^*$-resolution of
$F(\Omega^n(S))$. By Theorem \ref{grapplyF},
\[ (**)\quad\quad
\Ext_{\Lambda^*}^t(\Phi(F(\Omega^n_{\Lambda}(X))),-)\cong
\Ext_{\Lambda^*}^t(\Omega^n_{\Lambda^*}(\Phi(F(X))),-),
\] for $t>c$.
By (**) , dimension shift,  and since $\Phi$ commutes with $F$ and $\Omega$,
\[\Ext_{\Lambda^*}^t(\Phi(F(\Omega^n_{\Lambda}(X))),-)\cong
\Ext_{\Lambda^*}^{t+n}(\Phi(F(X)),-).\]
Let $S$ be a simple graded $\Lambda$-module belonging to $f$ and
let $S^*=F(S)$.  First we show that, using the above isomorphisms, if
$n>b+c+2$, then $F$ induces a monomorphism
\[ \Ext^n_{\Lambda}(\Phi(X),\Phi(S))\to
\Ext^n_{\Lambda^*}(\Phi(F(X)),\Phi(S^*)).
\]
We recall that $\Ext^n_{\Lambda}(\Phi(X),\Phi(S)) \cong \Ext^n_{\Gr(\Lambda)}(X,\oplus_{g\in G}S[g])$
and  that \linebreak
$\Ext^n_{\Lambda^*}(\Phi(F(X)),\Phi(S^*)\cong\Ext^n_{\Gr(\Lambda)}(F(X),\oplus_{g\in G}S^*[g])$.
Suppose $\alpha:P^n\to S[g]$ represents a nonzero element in
$\Ext^n_{\Gr(\Lambda)}(X,S[g])$.  It is now easy to see that
$F(\alpha)\colon F(P^n)\to S^*[g]$ is nonzero. By minimality of
$F(\cP^{\bullet})$ from $b+1$ on, $F(\alpha)$ represents a nonzero
element of $\Ext^n_{\Gr(\Lambda)}(F(X),S^*[g])$.

Having shown that if $n>b+c+2$, then $F$ induces an
monomorphism
\[ \Ext^n_{\Lambda}(X,S)\to
\Ext^n_{\Lambda^*}(F(X),S^*),
\]
we now show that $F$ induces an epimorphism. 
Since 
\[\Ext^n_{\Lambda}(\Phi(X),\Phi(S))\cong \Hom_{\Lambda}(\Phi(P^n),\Phi(S))\cong \Hom_{\Lambda_0}(\Phi(P^n/\br P^n),\Phi(S))\] and
\[\Ext^n_{\Lambda^*}(\Phi(F(X)),\Phi(S^*))\cong \Hom_{\Lambda^*}(\Phi(F(P^n)),\Phi(S^*))\]\[\cong \Hom_{\Lambda^*_0}(\Phi(F(P^n)/\br^* F( P^n)),\Phi(S^*)),\]
we conclude that the lengths of $\Ext^n_{\Lambda}(\Phi(X),\Phi(S))$ and 
$\Ext^n_{\Lambda^*}(\Phi(F(X),\Phi(S^*))$
are equal as modules over the commutative Artin ring $C$, over which both $\Lambda/\br$ and
$\Lambda^*/\br^*$ are both finite length modules.
Since $F$ induces a monomorphism, we conclude that $F$ is an isomorphism.

By taking direct sums over simple modules belonging to
$f$, the isomorphisms $\Ext^n_{\Lambda}(\Phi(X),\Phi(S))\to
\Ext^n_{\Lambda^*}(\Phi(F(X)),\Phi(S^*))$ induced by $F$ extends to an
isomorphism
\[
\Ext^n_{\Lambda}(X,\Lambda/\br f)\cong
\Ext^n_{\Lambda^*}(F(X),\Lambda^*/\br^*),
\]
Taking $X=(\Lambda/\br)f$ we obtain the isomorphism
\[
\Ext^n_{\Lambda}((\Lambda/\br)f,(\Lambda/\br) f)\cong
\Ext^n_{\Lambda^*}(\Lambda^*/\br^*,\Lambda^*/\br^*),
\]
 Since $F$ is an exact functor, the induced isomorphism
\[
\bigoplus_{n\ge b+c+2}\Ext^n_{\Lambda}(\Lambda/\br f,\Lambda/\br f)\cong
\bigoplus_{n\ge b+c+2}\Ext^n_{\Lambda^*}(\Lambda^*/\br^*,\Lambda^*/\br^*),
\] is an isomorphism of $\mathbb Z\times G$-graded rings (without identity) the
assertion about $\oplus_{n>b+c+2}\Ext^n_{\Lambda}((X,(\Lambda/\br) f)\cong
\oplus_{n>b+c+2}\Ext^n_{\Lambda^*}(F(X),\Lambda^*/\br^*)$ being a graded module isomorphism
follows.\end{proof}

We have the following consequence of the above proof.

\begin{prop}\label{pd-simple} 
Keeping the notation and hypothesis of Theorem \ref{grfirst-half},
let $X$ be a graded $\Lambda$-module.  Then $\pd_{\Lambda}(\Phi(X)) <\infty$ if and only
if $\pd_{\Lambda^*}(\Phi(F(X)) <\infty$.
\end{prop} 

\begin{proof}  From the proof of Theorem \ref{grfirst-half}, we see that
for every graded simple $\Lambda$-module $S$ belonging to $f$,
\[
\Ext_{\Lambda}^n(\Phi(X),\Phi(S))\cong \Ext_{\Lambda^*}^n(\Phi(F(X)),\Phi(F(S))),\]
for $n>b +c+2$.   But by our assumption on finitely generated resolutions,
and that $\id_{\Lambda}(\Lambda e)=b$,  we see that $\Ext_{\Lambda}^n(\Phi(X),-)=0$
if and only if  $\Ext_{\Lambda}^n(\Phi(X),\Phi(S)) =0$ for all graded simple modules  $S$ belonging to $f$.
Finally, $\Ext_{\Lambda}^n(\Phi(X),-)=0$ if and only if $\pd_{\Lambda}(X)\le n-1$.
\end{proof}

By combining Proposition \ref{grExt-G-S} and Theorem \ref{grfirst-half},
we obtain the desired result.

\begin{thm}\label{grG-to-ss-thm}
 Let $G$ be a group with identity element $\mathfrak e$
and $\Lambda=\oplus_{g\in G}\Lambda_g$ be
a properly $G$-graded ring in which graded idempotents lift.
Assume that every graded simple $\Lambda$-module has a finitely generated minimal
graded projective
$\Lambda$-resolution.  Let $\br$ denote the graded Jacobson radical of $\Lambda$.
Suppose that  $(e,f)$ is a suitable idempotent pair in
$\Lambda_{\mathfrak e}$ and set $\Lambda^*$ be the ring $f\Lambda
f$ and $\br^*=f\br f$. Assume that  $\pd_{\Lambda^*}(f\Lambda e)=c<\infty$, 
$\pd_{\Lambda}((\Lambda/\br)  e)=a<\infty$, and
that $\id_{\Lambda}((\Lambda/\br) e)=b<\infty$. Then, for $n>\max\{a,b+c+2\}$,
there are isomorphisms
$\Ext^n_{\Lambda}(\Lambda/\br,\Lambda/\br)\cong
\Ext^n_{\Lambda^*}(\Lambda^*/\br^*,\Lambda^*/\br^*)$ such that the induced
isomorphism
\[\bigoplus_{n>\max\{a,b+c+2\}}\Ext^n_{\Lambda}(\Lambda/\br,\Lambda/\br)\cong
\bigoplus_{n>\max\{a,b+c+2\}}\Ext^n_{\Lambda^*}(\Lambda^*/\br^*,\Lambda^*/\br^*)\]

Letting $\Delta = \oplus_{n>\max\{a,b+c+2\}}\Ext^n_{\Lambda}(\Lambda/\br,\Lambda/\br)$, if  $X$ is a graded $\Lambda$-module having a finitely 
generated projective resolution, then
\[
\bigoplus_{n>\max\{a,b+c+2\}}\Ext^n_{\Lambda}(\Phi(X),\Lambda/\br)\text{ and }
\bigoplus_{n>\max\{a,b+c+2\}}\Ext^n_{\Lambda^*}(\Phi(F(X)),\Lambda^*/\br^*)\]
are isomorphic as graded $\Delta$-modules.
\qed \end{thm}

\section{Applications}\label{applications}
We begin this section with a well-known result whose proof we
include for completeness.

\begin{prop}\label{large-gen}Let $R=R_0\oplus R_1\oplus
R_2\oplus\cdots$ be a finitely generated positively $\mathbb Z$-graded $C$-algebra where
$C$ is a commutative ring.  Let $N$ be a fixed positive
integer.  Then there
is a positive integer $D$ with $N< D$ such that the following holds.
\begin{enumerate}
\item[] If $j> D$ and $r\in R_j$, then
$r=\sum_ic_iu_{i,1}u_{i,2}\cdots u_{i,m_i}$, where $c_i\in C$ and each $u_{i,k}\in
R_{\ell}$, with $N\le \ell< D$.
\end{enumerate}
\end{prop}

\begin{proof}
Assume that $R$ can be generated over $C$ by homogeneous elements $x_1,\dots,x_n$ with each $x_i$
having degree at least $0$ and $L=\max\{\deg x_i\mid 1\le i\le n\}$.
Set $D=2LN$ and suppose $r\in R_j$ with $j>D$.  Then, by finite
generation, $r=\sum_ic_iy_{i,1}\cdots y_{i,t_i}$ where, for all $i,k$, $c_i\in C$,
$y_{i,k}\in\{x_1,\dots,x_n\}$ and $\sum_{k=1}^{t_i}\deg(y_{i,k})=j$,
for each $i$. Fix $i$ and write $y_{j}$ instead of $y_{i,j}$ and set
$t=t_i$. We see that \[D=2NL<j=\sum_{k=1}^{t}\deg(y_{k})\le Lt.\]
Hence $2N<t$. Write $t=AN+S$, where $0\le S<N$. For $i=1,\dots,A-1$,
set $u_i=y_{(i-1)N+1}y_{(i-1)N+2}\cdots y_{iN}$ and
$u_A=y_{(A-1)N+1} \cdots y_{AN}\cdot y_{AN+1}\cdots y_{t}$. It is
immediate that for $1\le i\le A$, $N\le \deg(u_i)< 2NL=D$.  This
completes the proof.

\end{proof}

We have some immediate consequences.

\begin{cor}\label{eventual}Let $R=R_0\oplus R_1\oplus
R_2\oplus\cdots$ be a positively $\mathbb Z$-graded
 ring such
that, $R_0$ is an Artin algebra over a commutative Artin ring $C$,
and, for $i\ge 0$, $R_i$ has finite length over $R_0$.
 Let $N$ be a fixed positive
integer.  Then $R$ is finitely generated as ring over $C$ if and only
if $T=\oplus_{i\ge N}R_i$ is finitely generated as a ring (without identity) over $C$.
\end{cor}

\begin{proof} Note that $R_0\oplus R_1\oplus \cdots\oplus R_{N-1}$ is of finite length over
$C$.  If $T$ is finitely generated over $C$, adding a finite numbers  generators of  
$R_0\oplus R_1\oplus \cdots\oplus R_{N-1}$ over $C$ to a 
set of  generators $T$ yields a finite generating set for $R$.

If $R$ is finitely generated as a ring over $C$, the proof of
Proposition \ref{large-gen} implies that $T$ is finitely generated,
by taking as a generating set, all products of the form
\[y_1y_2\cdots y_t,\]
where the $y_i$'s are elements of a finite generating set of $R$ and
$N\le t <2NL$, with $L$ being the maximum degree of the generators
of $R$.
\end{proof}

Before stating the main theorem of the section, we consider low terms
in resolutions of simple $\Lambda$- and $\Lambda^*$-modules.  More
precisely, suppose that $G$ is a group and that $\Lambda$ is a properly
$G$-graded ring in which graded idempotents lift.  Let $(e, f)$
be a suitable idempotent pair in $\Lambda$
and let $\Lambda^*=f\Lambda f$, $\br$ and 
$\br^*$ the graded Jacobson radicals of $\Lambda$ and $\Lambda^*$ respectively.
Assume all the conditions of Theorem \ref{grG-to-ss-thm}.
Let $S$ be a graded simple $\Lambda$-module and
$S^*=f\Lambda\otimes_{\Lambda}S$, viewed as a graded $\Lambda^*$-module.
Example \ref{ex41} shows that even if $S$ has a finitely generated graded
projective $\Lambda$-resolution, $S^*$ need not have a finitely generated 
graded projective $\Lambda^*$-resolution.
To remedy this situation,  we have the following result and its collorary.

\begin{prop}\label{gen-fp} Let $G$ be a group and $R=\oplus_{g\in G}R_g$ be a $G$-graded
ring.  Let $\cdots \to X^2\extto{d^2}X^1\extto{d^1}X^0\extto{d^0}M\to 0$ 
be an exact sequence of graded $R$-modules.  If, for all $n\ge 0$, $X^n$ has
a finitely generated graded projective $R$-resolution, then $M$ has a finitely
generated graded projective  $R$-resolution.
\end{prop}

\begin{proof}
For $j\ge 0$, let $X^{0,j}= X^j$ and $Y^{0,j}= \Im(d^j)$.  Note that
$Y^{0,0}=M$.  For each $j\ge 0$,  let 
\[\cdots \to P^{2,j}\extto{\delta^{2,j}}P^{1,j}\extto{\delta^{1,j}}P^{0,j}\extto{\delta^{0,j}
}X^{0,j}\to 0\]
be a finitely generated graded projective $R$-resolution of $X^{0,j}$.
For $i\ge 0$, define $X^{i,j}= \Im(\delta^{i,j})$.
Thus, for each $i\ge 0$, we have short exact
sequences
\[ 0\to X^{i+1,j}\to P^{i,j}\to X^{i,j}\to 0.\]

We inductively construct 
graded $R$-modules  $Y^{i,j}$ and finitely generated graded projective 
$R$-modules $Q^{i,j}$ such that
\begin{enumerate}
\item for each $i,j\ge 0$, there is a short exact sequence
$0\to Y^{i+1,j}\to Q^{i,j}\to Y^{i,j}\to 0$,
\item for $i\ge 0$ and $j\ge 1$, there is a short exact sequence
$0\to X^{i+1,j-1}\to Y^{i+1,j-1}\to Y^{i,j}\to 0$     and,
\item for $i\ge 1$ and $j\ge 0$, $Q^{i,j}= Q^{i-1,j+1}\oplus P^{i,j}$. 
\end{enumerate}

Once this is accomplished, splicing together the short exact sequences
$0\to Y^{i+1,0}\to Q^{i,0}\to Y^{i,0}\to 0$ we obtain a long exact sequence
\[\cdots\to Q^{2,0}\to Q^{1,0}\to Q^{0,0}\to  Y^{0,0}\to 0.\]  But $Y^{0,0}=M$ and
the result follows.

We have defined $Y^{0,s}$ and $P^{0,s}$ for all $s\ge 0$.   Set $Q^{0,i}=P^{0,i}$, for
all $i\ge 0$. 
We have exact sequences
\[ 0\to  Y^{0,s+1}\to X^{0,s}\to Y^{0,s}\to 0,\]
for all $s\ge $.  We also have exact sequences
$0\to X^{1,s}\to P^{0,s}\to X^{0,s}\to 0$ for all $s\ge 0$.  From these exact
sequences we obtain the following commutative diagram.

\[
\xymatrix{
&&0&0\\
0\ar[r]&Y^{0,s+1}\ar[r]&X^{0,s}\ar[u]\ar[r]&Y^{0,s}\ar[r]\ar[u]&0\\
&0\ar[r]&P^{0,s}\ar@{=}[r]\ar[u]& P^{0,s}\ar[r]\ar[u]&0\\
&0\ar[r]&X^{1,s}\ar[r]\ar[u]& Y^{1,s}\ar[r]\ar[u]& Y^{0,s+1}\ar[r]&0\\
&&0\ar[u]&0\ar[u],
}\]
where $Y^{1,s}$ is defined to be the kernel of the surjection $P^{0,s}$ to
$Y^{0,s}$.  Thus, we have defined $Y^{1,j}$ such that (1) holds for all
$i=0$ and $j\ge 0$ and (2) holds
 for all $j\ge 1$ and $i=0$.  For $i=0$, (3) vacuously holds.

Now consider $0\to X^{1,s}\to Y^{1,s}\to Y^{0,s+1}\to 0$
Using the exact sequences
$0\to Y^{2,s}\to Q^{1,s}\to Y^{1,s}\to 0$ and
$0\to X^{1,s+1}\to P^{0,s+1}\to X^{0,s+1}\to 0$ and using the Horseshoe
Lemma, we obtain the following commutative diagram
\[
\xymatrix{
&0&0&0\\
0\ar[r]&X^{1,s}\ar[u]\ar[r]&Y^{1,s}\ar[r]\ar[u]&Y^{0,s+1}\ar[u]\ar[r]&0\\
0\ar[r]&P^{1,s}\ar[r]\ar[u]&  P^{1,s}\oplus Q^{0,s+1}\ar[r]\ar[u]&Q^{0,s+1}\ar[r]\ar[u]&0\\
	0\ar[r]&X^{2,s}\ar[r]\ar[u]& Y^{2,s}\ar[r]\ar[u]& Y^{1,s+1}\ar[u]\ar[r]&0\\
&0\ar[u]&0\ar[u]&0\ar[u],
}\]  where $Y^{2,s}$ is the kernel of $ P^{1,s}\oplus Q^{0,s+1}\to Y^{1,s}$.  Let
$Q^{1,s}= P^{1,s}\oplus Q^{0,s+1}$.  It is immediate that (1) holds for all $j\ge 0$,
(2) holds for all $j\ge 1$ and $i=1$, and (3) holds for $i=1$ and all $j\ge 0$.

Continuing in this fashion, we define the $X^{i,j}$ and $P^{i,j}$ for all $i,j\ge 0$ satisfying
(1), (2), and (3).
\end{proof}

\begin{cor}\label{fg-resol}
 Let $G$ be a group with identity element $\mathfrak e$
and $\Lambda=\oplus_{g\in G}\Lambda_g$ be
a properly $G$-graded ring in which graded idempotents lift.
Suppose that  $(e,f)$ is suitable idempotent pair
and set $\Lambda^*$ be the ring $f\Lambda
f$.   Assume that, as a left $\Lambda^*$-module, $f\Lambda e$ has
a finitely generated graded projective resolution.
Let $M$ be a graded $\Lambda$-module having a finitely generated graded
projective $\Lambda$-resolution.  Then $F(M)$ has a finitely generated graded
projective $\Lambda^*$-resolution.
\end{cor}

\begin{proof}
Let $M$ be a graded $\Lambda$-module and let $\cP: \cdots\to P^1\to P^0\to M\to 0$
be a finitely generated graded projective $\Lambda$-resolution of $M$.
Applying the exact functor $F$, we get an exact sequence graded $\Lambda^*$-modules
 $F(\cP): \cdots\to F(P^1)\to F(P^0)\to F(M)\to 0$.
The result will follow if we show that each $F(P^n)$ has a finitely generated graded projective
$\Lambda^*$-resolution.
  For each $n\ge 0$,
set $P^n=P^n_e\oplus P^n_f$, where  $P^n_e$ belongs to $e$ and $P^n_f$ belongs to $f$.
By our hypothesis, $F(P^n_e)$ has a finitely generated graded projective $\Lambda^*$-resolution.
Since $F(P^n_f)$ is a graded projective $\Lambda^*$ module and since $F(P^n)=F(P^n_e)
\oplus F(P^n_f)$ we are done.
\end{proof}

We can state the main theorem of this section.  If $\Lambda$ is a ring, then let
$\gk(\Lambda)$ denote the Gelfand-Krillov dimension of $\Lambda$ and
$\gldim(\Lambda)$ denote the  (left) global dimension of $\Lambda$.

\begin{thm}\label{bigthm}
 Let $G$ be a group with identity element $\mathfrak e$
and $\Lambda=\oplus_{g\in G}\Lambda_g$ be
a properly $G$-graded ring in which graded idempotents lift.
Assume that every graded simple $\Lambda$-module has a finitely generated minimal
graded projective
$\Lambda$-resolution.  Let $\br$ denote the graded Jacobson radical of $\Lambda$.
Suppose that  $(e,f)$ is a  suitable idempotent pair and set $\Lambda^*$ be the ring $f\Lambda
f$ and $\br^*=f\br f$.
Suppose that $\pd_{\Lambda^*}(f\Lambda e)<\infty$, 
$\pd_{\Lambda}((\Lambda/\br)  e)<\infty$, and
that $\id_{\Lambda}((\Lambda/\br) e)<\infty$.   
Then the following hold.
\begin{enumerate}
\item Assume that $f\Lambda e$ has a finitely generated minimal
graded projective
$\Lambda^*$-resolution. 
The cohomology ring $E(\Lambda)$ is finitely generated over\linebreak
$\Ext^0_{\Lambda}(\Lambda/\br,\Lambda/\br)\cong \Hom_{\Lambda}(\Lambda/\br,\Lambda/\br)
\cong (\Lambda/\br)^{op}$ if and only if the cohomology ring $E(\Lambda^*)$ is finitely generated
as a $(\Lambda^*/\br^*)^{op}$-algebra. 
\item Assume that $\Lambda$ is $K$-algebra, where $K$ is a field and that $\Lambda/\br$ is a finite dimensional $K$-algebra.   Assume further that both $E(\Lambda)$ and $E(\Lambda^*)$
are finitely generated $K$-algebras.  Then
$\gk(E(\Lambda))
=\gk(E(\Lambda^*))$.
\item  We have that $\pd_{\Lambda}(S)<\infty$, for all
graded simple $\Lambda$-modules $S$ if and only if $\pd_{\Lambda^*}(S^*)<\infty$, for all
graded simple $\Lambda^*$-modules $S^*$.  
\end{enumerate}

\end{thm}

\begin{proof} Suppose $C$ is a commutative Artin algebra over which $\Lambda/\br$ has finite
length.  Note that if $S^*$ is a graded simple $\Lambda^*$-module, then there exists a 
graded simple $\Lambda$-module $S$ such that $S^*\cong F(S)$.    By Corollary \ref{fg-resol} and
our assumptions,  it
follows every graded simple $\Lambda^*$-module has a finitely generated graded projective
$\Lambda^*$-resolution.  In particular, for $n\ge 0$,  $\Ext_{\Lambda^*}^n(
\Lambda^*/\br^*,\Lambda^*/\br^*)$ has finite length over $C$. 
Part (1) follows from Theorems \ref{grfirst-half} and
\ref{grG-to-ss-thm}, Proposition \ref{large-gen}, and Corollary
\ref{eventual}.
Part (2) follows from the definition of Gelfand-Kirillov dimension and
Theorem \ref{grG-to-ss-thm}. 
For Part (3), the proof is basically
given in the proof of Proposition \ref{pd-simple}.

\end{proof}

Applying these results the Artin algebra case we get the following
Corollary.

\begin{cor}\label{bigcor}
Let $\Lambda$ be an Artin algebra.
Let $\br$ denote the graded Jacobson radical of $\Lambda$.
Suppose that  $(e,f)$ is a suitable idempotent pair in
$\Lambda$ and set $\Lambda^*$ be the ring $f\Lambda
f$ and $\br^*=f\br f$.
Suppose that $\pd_{\Lambda^*}(f\Lambda e)<\infty$, 
$\pd_{\Lambda}((\Lambda/\br)  e)<\infty$, and
that $\id_{\Lambda}((\Lambda/\br) e)<\infty$.   
Then the following hold.
\begin{enumerate}
\item 
The cohomology ring $E(\Lambda)$ is finitely generated over
$\Ext^0_{\Lambda}(\Lambda/\br,\Lambda/\br)\cong \Hom_{\Lambda}(\Lambda/\br,\Lambda/\br)
\cong (\Lambda/\br)^{op}$ if and only if the cohomology ring $E(\Lambda^*)$ is finitely generated
as a $(\Lambda^*/\br^*)^{op}$-algebra. 
\item Assume that $\Lambda$ is a finite dimensional $K$-algebra, where $K$ is a field. 
Assume further that both $E(\Lambda)$ and $E(\Lambda^*)$ are finitely generated
rings.  Then
$\gk(E(\Lambda))
=\gk(E(\Lambda^*)$.
\item  We have that $\gldim(\Lambda)$ is finite if and only if $\gldim(\Lambda^*)$ is finite.  
\end{enumerate}

\end{cor}

\begin{proof}  We take $G$ to be the trivial group and view $\Lambda$ as a graded algebra.
Then the grading is proper and graded idempotents lift.  Every (graded) simple $\Lambda$-module and
(graded) simple $\Lambda^*$-module has a finitely generated projective resolution, as does
$f\Lambda e$.  The result is now a direct consequence of Theorem \ref{bigthm}.
\end{proof}

\section{Concluding remarks and examples}\label{rmk-ex}

We begin this section with a discussion of the construction of $\Lambda^*=f\Lambda f$ in
case $\Lambda$ is a quotient of a path algebra. We keep the notation of 
Example \ref{Gfinite}; namely,
let $K$ be a field,
$\cQ$ be a finite quiver,  $G$ a group, $W\colon \cQ_1\to G\setminus
\{\mathfrak e\}$ be a
weight function, and $I$ a graded ideal in the path algebra
$K\cQ$ generated by weight homogeneous elements.
We also assume that $I$ is contained in the ideal of $K\cQ$ 
generated by the arrows of $\cQ$.   Setting $\Lambda=K\cQ/I$, 
the $G$-grading on $K\cQ$ obtained
from $W$ induces a proper $G$-grading on $\Lambda$ such that
graded idempotents lift.
 
To simplify notation, if $x\in K\cQ$, then we denote the element $x +I$ of
$\Lambda$ by $\overline{x}$.   We wish to describe
$\Lambda^*=f\Lambda f$, where $f =\sum_{v\in X}\overline{v}$
and $X$ is a subset of the vertex set $\cQ$.  Set
$e=\sum_{v\in \cQ_0\setminus X}\overline{v}$. We keep the
notation that $\br$ is the graded Jacobson radical of $\Lambda$ and
$\br^*=f\br f$.  Note that $\br$ is generated all elements
of the form $\overline a$, for $a\in\cQ_1$, $\Lambda^*$ has 
an $G$-grading induced from the $G$-grading on $\Lambda$,
and $\br^*$ is the graded Jacobson radical of $\Lambda^*$.  
Furthermore, $\Lambda/\br\cong \prod_{v\in \cQ_0}K$ and
$\Lambda^*/\br^*\cong \prod_{v\in X}K$

 We define the quiver $\cQ^*$ as follows.  Let 
$\cQ^*_0=X$. To define the set of arrows $\cQ^*_1$ 
consider set of paths $\cM$ in $\cQ$ such that $p\in \cM$ if
 $p$ is a path of length $n$, $n\ge 1$ in $\cQ$
such that
$p=u_1\stackrel{a_1}{\to}u_2\stackrel{a_2}{\to}u_3\to\cdots\to
u_{n}\stackrel{a_{n}}{\to}u_{n+1}$, with $u_i$ belonging to $e$ for
$i=2,\dots,n$ and $u_1$ and $u_{n+1}$ belonging to $f$. Note that a vertex
$u$ belongs to $f$ (resp., to $e$) just means $u\in X$ (resp., $u\not\in X$).  Then  $
\cQ^*_1=\{a_p \mid p\in \cM, p \text{ is a path from } u_1\text{ to }
u_{n+1}\}$.   A  path $p\in \cM$
is called a \emph{minimal $f$-path} and the arrow $a_p$ in $\cQ^*$ is called
the \emph{arrow in $\cQ^*$ associated to minimal $f$-path $p$}.
We note that if $a\colon u\to v$ is an arrow with $u$ and $v$ 
belonging to $f$, then $a$ is a minimal $f$-path.  It is also easy
to see that if $p$ is a path in $Q$ from vertex $u$ to vertex $v$
with $u$ and $v$ belonging to $f$,  then $p$ can be uniquely
written as a product of paths $p_1\cdots p_m$, where each $p_i$
is a minimal $f$-path.

We now turn our attention to relations.
 Let
$I^*$ be the ideal in $K\cQ^*$ generated as follows.  If
$r\in I$ is an  element with $r=urv$, where $u$ and $v$ are vertices belonging
to $f$ and $r=\sum_ic_ip_i$, where $c_i\in K$ and $p_i$ is a path
from $u$ to $v$, then we set $r^*$ to be $\sum_ic_ip_i^*$ where
$p_i^*$ is the path in $\cQ^*$ obtained from $p_i$ by replacing each
minimal $f$-subpath $p$ in $p_i$ by $a_p$.  Note that if
a minimal $f$ is in $I$, then associated arrow is in $I^*$.  We also
note that although $\cQ$ is a finite quiver and $\cQ^*_0$ is a
finite set, $\cQ^*$ may have an infinite number of arrows.    The next example
demonstrates this and that even if $I$ is an ideal in $K\cQ$, finitely generated 
by homogeneous elements, $I^*$ need not be finitely generated.   

\begin{Example}\label{ex41}{\rm  Let $\cQ$ be the quiver
\[\xymatrix{
\stackrel{u}{\circ}\ar_a[r]&\stackrel{v}{\circ}\ar@(r,u)[]^b\ar_c[r]&\stackrel{w}{\circ}
}\]
Take $e=v$ and $f=u+w$.  It is not hard to see that  each path of the form
$cb^na$, $n\ge 0$ is a minimal $f$-path and that these are the only minimal  
$f$-paths.  Hence,  $\cQ^*$ is the quiver with two vertices $u$ and $w$, and a 
countable number of arrows $a_{ca}, a_{cba}, a_{cb^2a},\dots$, each starting at
$u$ and ending at $w$.

Let $W\colon \cQ_1\to \mathbb Z_{> 0}$ by $W(a)=W(b)=W(c)=1$ and $I$ be the
ideal in $K\cQ$ generated by $b^2$.  Set $\Lambda=K\cQ/I$ and $\Lambda^*=K\cQ^*/I^*$.
Then $\Lambda^*=f \Lambda f$ and $I^*=fIf$.  We have  $I^*$ is generated by
$\{a_{cb^na}\mid n\ge 2\}$.   Note that both $\Lambda$ and $\Lambda^*$ are
Artin algebras.  Note that $\pd_{\Lambda}((\Lambda/\br) e)= \id_{\Lambda}((\Lambda/\br) e)
=\infty$, where $\br$ is the graded Jacobson radical
of $\Lambda$.  Moreover, $\gldim(\Lambda^*)=1$.  This example shows that the finiteness
of the projective or injective dimension of $(\Lambda/\br)e$ is necessary in Colloray \ref{bigcor}.

If we take $I=(0)=I^*$ above, then both $\Lambda =K\cQ$ and $\Lambda^*=K\cQ^*$ are hereditary algebras.   Hence Theorem \ref{grG-to-ss-thm} holds; in fact,
the $\Ext^n$'s are 0 for $n\ge 2$.  But $\Ext^1_{\Lambda^*}(\Lambda^*/\br^*,\Lambda^*/\br^*)$ is
infinite dimensional and  hence Theorem \ref{bigthm} fails.  Note that $f\Lambda e$ does not have
a finitely generated graded projective $\Lambda^*$-resolution.
}\end{Example}

We leave the proof of the following result to the reader.

\begin{prop}\label{newquiver} Keeping the notation above,
$\Lambda^*\cong K\cQ^*/I^*$.
\qed
\end{prop}

In the quiver case where $e$ is a idempotent associated to a vertex,
the next result gives a sufficient condition for exactness of the
functor $H$ where $H\colon \Gr(\Lambda^*)\to\Gr(\Lambda)$ given by
$H(X)=\Hom_{\Lambda^*}(f\Lambda, X)$ (see Section \ref{main result}).

\begin{prop}\label{Gexact}  Let $K\cQ/I$ be a finite dimensional $K$-algebra
where $K$ is field and $I$ is an admissible ideal in the path algebra
$K\cQ$; that is, for some $n\ge 2$, $J^n\subseteq I\subseteq
J^2$ where $J$ is the ideal generated by the arrows of $\cQ$.  Assume that $K\cQ$ is $G$-graded with the grading coming
from a weight function $W\colon \cQ_1\to G$ and that $I$ can be
generated by homogeneous elements.  Let $e$ be an idempotent element
of $K\cQ$ associated to a vertex $v$.  If $\pd_{\Lambda}((\Lambda/\br)e)\le 1$, then $\pd_{\Lambda^*}(f\Lambda e)\le 1$ and
$H$ is exact.
\end{prop}

\begin{proof}   By \cite{ILP}, there is no loop at $v$.  Let $e$ be
the idempotent in $K\cQ$ associated to the vertex $v$ and let
$f=1-e$.
 Consider the short exact sequence $0\to \br e\to \Lambda
e\to (\Lambda/\br)e\to 0$. Applying the functor $F$, we obtain
\[\xymatrix{
0\ar[r]&f\Lambda\otimes_{\Lambda}\br e\ar[r]\ar[d]^{\cong}&
f\Lambda\otimes_{\Lambda}\Lambda
e\ar[r]\ar[d]^{\cong}&f\Lambda\otimes_{\Lambda}(\Lambda/\br)e\ar[r]\ar[d]^{\cong}&0\\
&f\br e&f\Lambda e&0 }\] It follows that $f\Lambda e\cong f\br e$.
Since $\pd_{\Lambda}(\Lambda/\br)e\le 1$, $\br e\cong \oplus \Lambda
w$  where the direct sum runs over the arrows $v\to w$ in $\cQ$   and
$w$ belongs to $f$, and where $\Lambda
w$ is the projective $\Lambda$-module associated to the
vertex $w$.  Since each $w$
belongs to $f$, it follows that $f\Lambda w =f\Lambda fw=\Lambda^*w$,
which is a projective $\Lambda^*$-module.  Thus $f\Lambda e)$ is
a projective $\Lambda^*$-module and by the remark after Proposition
\ref{FGbasic}, $H$ is exact.
\end{proof}

 Let $K\cQ/I$ be a finite dimensional $K$-algebra
where $K$ is field and $I$ is an admissible ideal in the path algebra
$K\cQ$.  Assume that $K\cQ$ is $G$-graded with the grading coming
from a weight function $W\colon \cQ_1\to G$ and that $I$ can be
generated by homogeneous elements.  Let $e$ be an idempotent element
of $K\cQ$ associated to a vertex $v$.  As usual let $f=1-e$.
It is well known that $\pd_{\Lambda}(\Lambda/\br)e\le 1$ if and only
if there exists a
uniform set $\rho$ of generators of $I$ such that $gv=0$ for all $g\in\rho$, where an
element $r\in K\cQ$ is \emph{uniform} if there exist vertices $u$ and
$v$ in $\cQ$ such that $r=urv$.  Thus,
if there exists a
uniform set $\rho$ of generators of $I$ such that $g v=0$ for all $g\in \rho$ and
if $\id_{\Lambda}(\Lambda/\br e)<\infty$ , then Theorems
\ref{grG-to-ss-thm} and \ref{bigthm} hold. 

The next result gives sufficient conditions so that
$pd_{\Lambda^*}(f\Lambda e)<\infty$.  

\begin{prop}\label{pd-finite}
 Let $G$ be a group
and $\Lambda=\oplus_{g\in G}\Lambda_g$ be
a properly $G$-graded ring in which graded idempotents lift.
Suppose that  $(e,f)$ is a suitable
 idempotent pair
and set $\Lambda^*$ be the ring $f\Lambda
f$.  
  Suppose that 
$0\to P^n\to\cdots\to P^1\to P^0\to(\Lambda/\br)e\to 0$
is a minimal graded 
projective $\Lambda$-resolution of $(\Lambda/\br)e$ and 
that each $P^i$, for $i\ge 1$, is a direct sum of indecomposable
projective $\Lambda$-modules
of the form $\Lambda w$ with $w$ a vertex belonging to $f$.  Then
$\pd_{\Lambda^*}(f\Lambda e)<\infty$.
\end{prop}

\begin{proof}
Note that $f\Lambda\otimes_{\Lambda}(\Lambda/\br)e=0$ and
that $P^0\cong \Lambda e$; so that $f\Lambda\otimes_{\Lambda}\Lambda e \cong f\Lambda e$. We
see that the result follows by tensoring the projective resolution of
$(\Lambda/\br)e$ with $f\Lambda\otimes_{\Lambda}-$.
\end{proof}

\begin{Example}{\rm
We end with a nontrivial class of examples where the hypothesis of main 
theorems of the paper hold.  
Let $K$ be a field and $\Delta$ and $\Sigma$ be finite 
dimensional $K$-algebras.
Suppose that $A$ is $K$-$\Sigma$-bimodule,  $B$ is $\Delta$-$K$-bimodule
and $C$ is a $\Delta$-$\Sigma$-bimodule.   Let
\[
\Lambda =\left(\begin{array}{ccc}K&A&0\\
0&\Sigma&0\\
B&C&\Delta
\end{array}\right),\] 
where the ring operations are given by matrix addition and multiplication.
Set $e=\left(\begin{array}{ccc}1&0&0\\0&0&0\\0
&0&0\end{array}\right)$ and
$f=\left(\begin{array}{ccc}0&0&0\\0&1&0\\0
&0&1\end{array}\right).$  Note that
$f\Lambda f=
\left(\begin{array}{cc}
\Sigma&0\\
C&\Delta
\end{array}\right)$.

The reader may verify that if $\pd_{\Delta}B<\infty$ and
$\id_{\Sigma}(A)<\infty$, then $\pd_{\Lambda}(\Lambda/\br)e <\infty$,
$\id_{\Lambda}(\Lambda/\br)e <\infty$, and $\pd_{f\Lambda f}(f\Lambda
e)<\infty$.  Note that  $C$ is an arbitrary finite dimensional bimodule. Thus if  $\pd_{\Delta}B<\infty$  and
$\id_{\Sigma}(A)<\infty$,
then, taking $G=\{\mathfrak e\}$ to be the trivial group,
  Theorem \ref{grG-to-ss-thm} and  
Corollary \ref{bigcor} hold.
}
\end{Example}

\end{document}